\documentclass{amsart}
\usepackage{amsmath,amssymb,amsfonts}

\newtheorem{theorem}{Theorem}
\newtheorem{lemma}{Lemma}
\newtheorem{prop}{Proposition}
\newtheorem{cor}{Corollary}
\theoremstyle{plain}

\newtheorem{definition}[theorem]{Definition}
\newtheorem{remark}[theorem]{Remark}
\newtheorem{example}[theorem]{Example}

\begin{document}

\keywords{non-local operator, jump process, derivation property, Liouville property}
\subjclass[2000]{Primary {\bf 42B20}; Secondary {\bf 47G20, 31B05, 31C25}}


\newcommand{\dis}{\displaystyle}
\newcommand{\form}{{\mathcal E}}
\newcommand{\dom}{{\mathcal F}}
\newcommand{\bu}{\bullet}
\newcommand{\vareps}{\varepsilon}
\newcommand{\al}{{\alpha}}

\title{$L^p$-Liouville property for non-local operators}

\author{Jun Masamune}

\address{Department of 
Mathematics and Statistics,
Penn State Altoona,
3000 Ivyside Road, Altoona PA 16601, USA}

\email{jum35@psu.edu and t-uemura@kansai-u.ac.jp} 

\author{Toshihiro Uemura}

\address{Department of Mathematics, Faculty of Engineering Science,
Kansai University, Suita,Osaka 564-8680, Japan}

\email{t-uemura@kansai-u.ac.jp} 

\begin{abstract}
The $L^p$-Liouville property of a non-local operator $\mathcal{A}$
is investigated via the associated Dirichlet form $(\form,\dom)$.
We will show that any non-negative continuous $\form$-subharmonic function 
$f \in \dom_{\rm loc} \cap L^p$ are constant under a quite mild assumption 
on the kernel of $\form$ if $p \ge 2$. 
On the contrary, if $1<p<2$, we need an additional assumption: 
either, the kernel has compact support; or $f$ is H\"older continuous
\end{abstract}

\date{}
\maketitle{}

\section{Introduction}

The original Liouville theorem says that any bounded harmonic functions on 
an Euclidean space is identically constant. 
The study of the 
Liouville property was boosted by the 
celebrated $L^p$-Liouville theorems on Riemannian manifolds
by Andreotti-Vezentini \cite{AV65} and 
Yau \cite{Y76}, and since then, 
many generalizations and various types of Liouville properties have been 
considered (\cite{N85, T03, M05, PRS05}, etc and the references within). 
The problem of the 
Liouville property can be formulated via the Dirichlet integral; 
namely, to find the conditions 
such that a 
function $f$ satisfying
\begin{equation} \label{Eq;1.1}
(\nabla f, \nabla \psi) \le 0\mbox{ for all $\psi \in C^\infty_K$ with 
$\psi \ge 0$}
\end{equation}
is constant. 

A general frame work, which generalizes the classical Dirichlet integral on
a Riemannian manifold, uses Dirichlet form theory.  
Biroli and Mosco \cite{BM95} defined the subharmonic functions 
associated to a Dirichlet form $\form$, which we call \emph{$\form$-subharmonic} 
by replacing the Dirichlet integral in (\ref{Eq;1.1}) by $\form$ (see 
Definition \ref{defi;subharmonic} in Section \ref{DPSFO}). 
Subsequently, Sturm \cite{S94} generalized the $L^p$-Liouville properties 
of Riemannian manifolds \cite{AV65, Y76} 
to a general $\form$-subharmonic function of a (strong) 
local Dirichlet form. 

The 
$L^p$-Liouville property for 
local operators 
has been successfully developed. However, as far as the authors 
know, 
there are no results about the 
$L^p$-Liouville property for 
\emph{non-local operators}. 
Of course, 
non-local operators 
appear naturally in many areas in mathematics; such as the theories 
of pseudo-differential operators, stochastic processes (in particular, jump process), 
Dirichlet forms (see \cite{S70, FOT94, J01, JS01, B04, L72}, and the references therein).

In this article, we are interested in the $L^p$-Liouville property of a non-local 
operator $\mathcal{A}$ on the Euclidean space ${\mathbf{R}}^d$, given 
by 
\[
{\mathcal A} f (x) = \int_{y\not=x} \left(f(y) - f(x) - \nabla 
f (x)  (y-x) 1_{|x-y|<1}\right) \mu(x,dy), \ \ x \in {\mathbf{R}}^d,
\]
where $\mu(x,dy)$ is a (jump) kernel on ${\mathbf{R}}^d\times{\mathcal B}({\mathbf{R}}^d)$; 
namely, $B \mapsto \mu(x,B)$ is a positive measure  on 
${\mathcal B}({\mathbf{R}}^d\backslash \{x\})$ for each fixed $x \in {\mathbf{R}}^d$ and 
$x \mapsto \mu(x, B)$ is a Borel measurable function for every $B \in 
{\mathcal B}({\mathbf{R}}^d \backslash \{x\})$.  
Let $m(dx)$ be a positive Radon measure on ${\mathbf{R}}^d$ with full support such 
that $\mu(x,dy)m(dx)$ is symmetric; that is, $\mu(x,dy)m(dx)=\mu(y,dx)m(dy)$.

Let $C^{\rm lip}_K({\mathbf{R}}^d)$ be the space of all uniformly Lipschitz continuous 
functions on ${\mathbf{R}}^d$ with compact support. We define a quadratic form 
$\form$ as
\[ \dis{
\form (f,g)=\iint_{x \neq y} \bigl(f(x)-f(y)\bigr) \bigl(g(x)-g(y)\bigr) 
\mu(x,dy)m(dx) }
\]
for $f,g \in C^{\rm lip}_K({\mathbf{R}}^d)$. Since $(\form,C^{\rm lip}_K({\mathbf{R}}^d))$ 
is closable (Lemma \ref{Lem;2.1} and \cite{U07b}),  its closure $(\form, \dom)$ 
is a regular \emph{non-local} symmetric Dirichlet form, and therefore, 
there exists the associated symmetric Hunt process of \emph{pure jump type}
by Fukushima theorem.
We will investigate the $L^p$-Liouville property of $\mathcal{A}$ via the 
related non-local Dirichlet form $(\form,\dom)$.

Throughout the paper we make 
the following assumptions:
\begin{enumerate}
\item[(i)] There exist $R>3r>0$ such that 
$\mu(x, B(x,R)\backslash B(x,r))>0$ for every 
$x\in {\mathbf{R}}^d$.
\item[(ii)] \[
\dis{ M_\alpha = \sup_{x \in \mathbf{R}^d} \int_{y \neq x} 
\left( 1 \wedge |x-y|^\alpha \right) \mu(x,dy)<\infty}
\]
for some $0 \le \alpha <2$ depending on each settings of  the problems.
\end{enumerate}
Since the inequalities $|x-y|^2\le |x-y|^{\beta} \le |x-y|^{\alpha}$ hold 
for $|x-y|<1$ and $\alpha\le\beta\le 2$, we note that  $M_2 \le M_{\beta} 
\le M_{\alpha}$ for $\alpha \le \beta\le 2$. 
These assumptions are so general that 
they are satisfied by most of the typical examples; 
for instance, symmetric $\alpha$-stable processes, symmetric 
stable-like processes, and many other 
symmetric L\'evy processes, which we will present in Section \ref{E}. 

\smallskip
The principal purpose of the present article is to show
the following $L^p$-Liouville property of a non-local operator:

\bigskip
\noindent
{\bf Main theorem (Theorem \ref{MT1} and Theorem \ref{MT2})}
{\it Assume one of the following conditions:
\begin{enumerate}
\item[{\rm (i)}] $2\le p<\infty$ and $M_q<\infty$ with the conjugate number 
$q$ of $p$; that is, $1/p+1/q=1$, or
\item[{\rm (ii)}] $1<p<2$, $M_1<\infty$, and for some $R>0$, 
the measure $\mu(x, \bullet)$ is supported in the ball $B(x,R)$ for every 
$x \in {\mathbf{R}}^d$.
\end{enumerate}
Then any non-negative $\form$-subharmonic function 
$f \in C \cap \mathcal{F}_{\rm loc} \cap L^p$ is identically constant.}

\bigskip
Here, $\dom_{\rm loc}$ is the space of measurable functions $f$
such that for any compact set $K \subset \mathbf{R}^n$
there exists $g \in \dom$ with $f=g \ m$-a.e. on $K$ (see \S 2).
We stress that the 
$L^2$-Liouville property always holds for $M_2<\infty$.

Let us compare the Main result with the classical 
$L^p$-Liouville property of local operators \cite{S94}.
Recall that the associated process to a local operator
is diffusion and it has not jumps. 

The significant difference occurs when 
$1<p<2$, and 
in this case, we need an additional assumption on the measure $\mu(x, dy)$.
Let us explain why this happens. Roughly speaking, if $p \ge2$ 
we may conduct the proof as in the classical proofs for the local operator by 
developing and applying some techniques for a non-local operator.
The key is to use $\chi f^{p-1}$, where $\chi$ is a cut-off function, 
as a test function.
On the other hand, if $p<2$, then $\chi (f + \epsilon)^{p-1}$ with $\epsilon>0$ 
can be used as a test function for a local operator;
however, this is not true for a non-local operator. 
As we have stated above, the construction of the test function is the key in the proof.
To overcome this difficulty, we will demonstrate that 
under the assumption such that the support $\mu(x, dy)$ is in $B(x,R)$ with (uniform) 
$R>0$ for every $x \in {\mathbf{R}}^d$; for instance, the truncated kernel 
${\bf{1}}_{|x-y|\le R} \mu(x,dy)$, then $\form (f, \chi (f + \epsilon)^{p-1})$ 
makes sense. In this way, we can continue and complete the proof for $1<p<2$.
Recall that this assumption means that the associated process has only small 
jumps, and in this sense, it looks more similar to a diffusion. 

We will discuss another way to ensure the convergence of $\form (f, \chi  f^{p-1})$ with 
$1<p<2$; that is, is to restrict the harmonic functions to the class of H\"older continuous functions. 
Indeed, we will show that $L^p$-Liouville property with $1<p<2$ 
in the class of H\"older continuous 
harmonic functions holds true for $\mu$ which has both small and big jumps:

\begin{theorem}[Theorem \ref{MT2}] \label{theorem;1.1} Assume $1<p<2$ and 
$M_1<\infty$. Let $f$ be a non-negative $\form$-subharmonic function 
which is H\"older continuous  
with 
H\"older exponent $\gamma$ with $1/p \le \gamma \le1$. In addition, 
if $f$ belongs to $C^{\gamma} \cap L^p$, then it is identically constant. 
\end{theorem} 

Therefore, in (ii) in the Main theorem, the corresponding process has only small 
jumps, and the class of subharmonic functions is the considerable large one; 
and in Theorem \ref{theorem;1.1}, the functions are H\"older continuous, 
but there are no assumptions on the kernel, 
i.e.\ the associated process may have both small and big jumps.

In the technical contribution, we consider the following. One of the most 
essential ingredients in the proof  is the derivation property.
Clearly, in general, a non-local operator does not satisfy the 
derivation property, however, we establish the following integral version, 
which is interesting by its own and is sufficient 
for our purpose: 

\begin{theorem}[Integral Derivation Property: Proposition \ref{prop:deri}]
Let $1 \le p<\infty$ and $q$ be its  conjugate; that is, $1/p + 1/q =1$. 
If $M_{\alpha} < \infty$ with some $0<\alpha\le2$, then
\[
\int \Gamma (f,gh)(x)\, m(dx)= \int g(x) \Gamma(f,h)(x)\, m(dx) + \int h(x) 
\Gamma (f,g)(x)\, m(dx)
\]
for all $f \in \dom_{\rm loc} \cap L^p$, $g\in \dom_{\rm loc} \cap L^q$, and 
$h \in C_K^{\rm lip}$.  Here 
\[
\Gamma (f,g) (x) = \int_{ y \neq x} \bigl(f(x)-f(y)\bigr)
\bigl(g(x)-g(y)\bigr) \mu(x,dy) 
\quad  x \in {\mathbf{R}}^d.
\]
\end{theorem}
A weaker version of this result can be found in \cite{CKS87} that is 
not sufficient to prove our Main theorem (see the remark in Section 
\ref{DPSFO}). 

\smallskip
Since our results lie in the intersection of Analysis, Potential theory,
and Stochastic Analysis, it is desirable to characterize the 
$\form$-superharmonic function as a potential or as an excessive function. 
Let us briefly discuss these relationships (see the remark in Section 
\ref{DPSFO}). It can be shown that $f$ is $\form$-superharmonic if and only 
if it is a potential, and it is known that if additionally $f$ belongs to the 
domain of the Dirichlet form, then $f$ is $\form$-superharmonic if and only 
if it is excessive \cite{FOT94}. In particular, these three conditions are 
equivalent if $\form$ is local \cite{S94}, which is still unknown for a 
non-local case. Moreover, the equivalence of the $L^\infty$-Liouville property 
and the recurrence property of the associated Markov process is known for a 
(strong) local Dirichlet form  (\cite{S94, K00}), 
which is still open for a non-local case.

Finally, let us point out that in order to avoid the technical complications, 
we restrict ourselves to the Euclidean space, but our results extend with 
obvious modifications to a locally compact metric measure space $X$, 
provided that 
any ball is relatively compact. 
This space $X$ includes a complete Riemannian manifold, a complete 
sub-Riemannian manifold, and a complete weighted manifold as examples.
Therefore, our results may depend on the metric structure but not on
the topological nor the volume growth of the underlying space.
This forms a strong contrast with the conservation property of 
a jump process \cite{MU}, which heavily depends on the volume growth.

We organize the article as follows: In Section \ref{DPSFO}, we recall some 
preliminary results. In Section \ref{Section;IDP}, we prove the integral 
derivation property (Proposition \ref{prop:deri}) and its Corollaries.
We prove the $L^p$-Liouville property in Section \ref{LP} for $p\ge 2$, 
and for $1<p<2$ in Section \ref{Section;LPII}. Finally, in Section \ref{E}, 
we present some examples.

\section{Preliminaries} \label{DPSFO}
This section contains preliminary results. 
Let us first establish some notations.

Denote by $L^0({\mathbf{R}}^d;m)=L^0({\mathbf{R}}^d)$ the space of all 
measurable functions on ${\mathbf{R}}^d$.  
If ${\mathcal K} \subset L^0({\mathbf{R}}^d)$, then ${\mathcal K}_K$ 
denotes the space of all functions $f$ in ${\mathcal K}$ 
with compact support. Set ${\mathcal K}_{K,p}={\mathcal K}_K \cap 
L^p({\mathbf{R}}^d;m)$ for $1\le p\le \infty$. 
We say that $f$ is locally in $\dom$ ($ u \in \dom_{\rm loc}$ in notation)
if for any relative compact set $O$ there exists a function $g \in \dom$ such that
$f=g$ $m$-a.e.\ on $O$.
We shall often suppress ${\mathbf{R}}^d$ and $m$ and simply refer to
$C$, $L^p$, etc as $C({\mathbf{R}}^d)$, $L^p({\mathbf{R}}^d;m)$, etc, respectively. 

Since $\form$-subharmonic functions are 
defined in a weak sense, it is crucial to find the suitable class of tests functions. 
In fact, $\form (f, g)$ may diverge for $f \in \dom_{\rm loc}$ and $g \in \dom \cap C_K$, 
so this problem is not trivial like the local case. 
In Lemma \ref{Lem;2.2}, we will show that  in addition, if $f$ belongs to $L^p$, then 
$\form (f,g) < \infty$.

\smallskip
Let us recall
\begin{lemma}[Example 1.2.4 \cite{FOT94} and \cite{U04}] \label{Lem;2.1}
If $M_2 < \infty$, then 
$(\form,C_K^{\rm lip})$ is a closable Markovian form on $L^2$. Therefore, there exists a 
symmetric Hunt process associated to the Dirichlet form $(\form,\dom)$.
\label{lemma:regularity}
\end{lemma}

\bigskip
Henceforth we assume $M_{\alpha} < \infty$ for some $0<\alpha \le 2$.
In the next lemma, we extend $\form$ to 
$\dom_{{\rm loc},p} \times \dom_{K,q}$ or $\dom_{K,q}\times 
\dom_{{\rm loc},p}$, where $1 \le p < \infty$ and $q$ is the conjugate number of $p$; 
$1/p+1/q=1$. 
Note that $C_K \cap \dom \subset \dom_{K,p}$ for any $1\le p\le \infty$.


\begin{lemma}\label{Lem;2.2}  
Let $1\le p<\infty$ and $q$ be its conjugate; that is, $1/p+1/q=1$. 
If $M_{\alpha} < \infty$ with some $0<\alpha\le 2$, then
\[
\iint_{x\not=y} \bigl| f(x)-f(y) \bigr|  
\bigl| g(x)-g(y) \bigr| \, \mu(x,dy)m(dx) <\infty\]
for any $f\in \dom_{{\rm loc},p}$ and $g \in \dom_{K,q}$.
Therefore, the integral 
\[
\form(f,g) =\form(g,f)= \iint_{x\not=y} 
\bigl(f(x)-f(y)\bigr)\bigl(g(x)-g(y)\bigr) \mu(x,dy)m(dx) 
\]
makes sense for all $f\in \dom_{{\rm loc},p}$ and $g \in \dom_{K,q}$.
\label{lemma:integral}
\end{lemma}
\begin{proof}
Let $f\in \dom_{{\rm loc},p}$ and $g \in \dom_{K,q}$. 
Let $R,r>0$ be such that $R-r\ge 1$ and 
${\rm supp}[g]\subset B(r)$.   
Take a function $f_R\in \dom$ such that $f=f_R$ $m$-a.e. on $B(R)$.
Using the symmetry of the measure $\mu(x,dy)m(dx)$, it follows:
\begin{align*}
   &\dis{\iint_{x \neq y} |f(x)-f(y)|  |g(x)-g(y)| \, \mu(x,dy)m(dx)}  \\
=&\dis{ \iint_{B(R)\times B(R)\backslash D} |f(x)-f(y)|  |g(x)-g(y)| \, \mu(x,dy)m(dx)}  \\
&+\dis{ \iint_{B(R)\times B(R)^c} |f(x)-f(y)|  
|g(x) | \, \mu(x,dy)m(dx)}  \\ 
&+\dis{ \iint_{B(R)^c\times B(R)} |f(x)-f(y)|  |g(y) | \, \mu(x,dy)m(dx)}  \\ 
=&\dis{ \iint_{B(R)\times B(R)\backslash D} |f(x)-f(y)|  
|g(x)-g(y)| \, \mu(x,dy)m(dx)}  \\ 
&+2\dis{ \iint_{B(R)\times B(R)^c} |f(x)-f(y)|   |g(x) | \, \mu(x,dy)m(dx)}  \\
 =&: {\rm (I)}+2{\rm (II)}.
\end{align*}
First we estimate (I). Since $f=f_R$ on $B(R)$ and $f_R\in \dom$, 
\begin{align*}
{\rm (I)} =&\dis{ \iint_{B(R)\times B(R)\backslash D} |f_R(x)-f_R(y)|  
|g(x)-g(y)| \, \mu(x,dy)m(dx) } \\
\le& \dis{ \sqrt{ \iint_{B(R)\times B(R)\backslash D} 
\bigl(f_R(x)-f_R(y)\bigr)^2 \mu(x,dy)m(dx)}  } \\
& \dis{ \times \sqrt{ \iint_{B(R)\times B(R)\backslash D} 
\bigl(g(x)-g(y)\bigr)^2 \mu(x,dy)m(dx)}  } \\
\le& \dis{ \sqrt{\form(f_R,f_R)} \sqrt{\form(g,g)} 
<\infty.} \\
\end{align*}

Next, since the support of $g$ is included in $B(r)$, 
applying the H\"older inequality, we have that
\begin{align*}
{\rm (II)} =& \dis{ \iint_{B(r)\times B(R)^c} |f(x)-f(y)|  |g(x) | \, \mu(x,dy)m(dx)}  \\  
\le& \dis{ \left( \iint_{B(r)\times B(R)^c} |f(x)-f(y)|^p \mu(x,dy)m(dx) \right)^{1/p} } \\
& \dis{ \times \left( \iint_{B(r)\times B(R)^c}  |g(x)|^q \mu(x,dy)m(dx) \right)^{1/q}}  \\ 
\le& \dis{ \left( 2^{p-1} \iint_{B(r)\times B(R)^c} 
\bigl(|f(x)|^p+|f(y)|^p\bigr) \mu(x,dy)m(dx) \right)^{1/p} } \\ 
& \dis{ \times \left( \int_{B(r)} |g(x)|^q \int_{B(R)^c} 
\mu(x,dy)m(dx) \right)^{1/q}}  \\ 
\le& \dis{ \left( 2^p \int_{B(r)} |f(x)|^p \int_{B(R)^c} 
\mu(x,dy)m(dx) \right)^{1/p} ||g||_{L^q}  M_{\alpha}^{1/q} } \\  
\le& \dis{ 2M_{\alpha} ||f||_{L^p} ||g||_{L^q}<\infty.}
\end{align*}
Here we used the inequality: $(a+b)^p \le 2^{p-1}\bigl(a^p+b^p\bigr)$ for 
$a\ge 0$, $b\ge 0$, and $1\le p<\infty$ in the second inequality.
\end{proof}

\begin{remark} If  $1\le p_1 <p_2 < 2 < p_3<\infty$ and $q_i \, 
(i=1,2,3)$ are the respective conjugate numbers; $1/p_i+1/q_i=1$, then
$1< q_3 < 2 < q_2 <q_1 \le \infty$. It is known that $L_K^p \subset L_K^2 \subset L^{p'}_K$
for $1\le p' < 2 < p$.  So it is clear that $\dom_K \subset L^2_K \subset L^{p'}_K$ and
$\dom \cap L^2_K=\dom_K$, because $\dom$ is the closure of 
$C_K^{\rm lip}$ with respect to 
$\bigl(\form(\cdot,\cdot)+||\cdot||_{L^2}^2\bigr)^{1/2}$.
Thus we find that 
$$
\dom_{K,q_1} \subset \dom_{K,q_2} \subset 
\dom_{K,2} \, \bigl(=\dom_K\bigr)
$$
and 
$$
\dom_{K,q_3}=\dom_K \cap L^{q_3} =(\dom \cap L^2_K)\cap L^{q_3}
=\dom \cap (L^2_K \cap L^{q_3})= \dom \cap L^2_K=\dom_{K,2}.
$$
This means $\dom_{K,q} = \dom_K$ for every $2\le p<\infty$ where $q$ is the 
conjugate number of $p$.

\label{remark:domain}
\end{remark}

\bigskip
We define the ``carr\'e du champ operator $\Gamma$ associated with ${\mathcal A}$'' 
(see \cite{BH91}, \cite{U07}) as
\[
\Gamma(f,g)(x):=\int_{y\not=x} \bigl(f(x)-f(y)\bigr)
\bigl(g(x)-g(y)\bigr)\mu(x,dy), \quad x\in {\mathbf{R}}^d
\]
for any pair $f\in \dom_{{\rm loc},p}$ and $g\in \dom_{K,q}$; or 
$f\in \dom_{K,q}$ and $g\in \dom_{{\rm loc},p}$ with $1\le p<\infty$ and 
its conjugate $q$. Then $\Gamma(f,g) \in L^1$ for such $f$ and $g$ by the 
previous lemma.

Set 
${\mathcal C}=C_K \cap \dom$ and 
\[
{\mathcal H}
=\Bigl\{ u\in L^0:\ \form(u,\varphi) \ \mbox{ makes  
sense  for  all  $\varphi \in {\mathcal C}$} \Bigr\}.
\]
By Lemma \ref{lemma:regularity}, Lemma \ref{lemma:integral}, 
Theorem 1.5.2 in \cite{FOT94}, and taking into account that ${\mathcal C} 
\subset \dom_{K,p}$ for $1\le p< \infty$, it follows that 
\[
\dom \subset \dom_e \subset {\mathcal H} \quad {\rm and} \ \ 
\bigcup_{1\le p<\infty} \left(\dom_{\rm loc} \cap L^p({\mathbf{R}}^d;m)\right) 
\subset {\mathcal H}.
\]
Here $\dom_e$ is the extended Dirichlet space of $\dom$ (see 
\cite{FOT94}). Different from the local (or second order elliptic 
differential operators) case, the form $\form$ or $\Gamma$ can not be 
extended to $\dom_{\rm loc}$ (see \S 3.2 in \cite{FOT94}). Therefore, in general, $\dom_{\rm loc}$ does not 
contain ${\mathcal H}$ and vice versa.

\begin{definition}[$\form$-Subharmonic functions] \label{defi;subharmonic}
A function $f \in {\mathcal H}$ is 
\emph{$\form$-subharmonic} (\emph{$\form$-superharmonic}, respectively) if 
\[
\form(f,\varphi) \le 0 \quad \quad 
(\form(f,\varphi)\ge 0, \mbox{respectively})
\]
for any $\varphi \in {\mathcal C}$ with $\varphi\ge 0$.
Moreover, $f\in {\mathcal H}$ is $\form$-{\it harmonic} 
if it is both $\form$-subhamonic and $\form$-superharmonic.
\end{definition}

\begin{remark} Let us discuss the relationship between $\form$-harmonic functions, 
potentials, and excessive functions.
\begin{itemize} 
\item[{\sf (1)}] As in \cite{Deny} or the proof of Theorem 2.2.1 
in \cite{FOT94}, we see that $f\in {\mathcal H}$ is 
$\form$-superharmonic if and only if it is a 
{\it potential}; that is, there exists a positive Radon measure 
$\mu$ on ${\mathbf{R}}^d$ such that
\begin{equation}
\label{potential}
\form(f,\varphi)=\int_{{\mathbf{R}}^d} \varphi(x) \mu(dx)
\end{equation}
for every $\varphi \in C_K^{\rm lip}$.  
In fact, let $I(v)=\form(f, v)$ for $v \in C_K^{\rm lip}$. Note that 
$C_K^{\rm lip}\subset {\mathcal C}$. Consider any compact set $K$ and choose 
a nonnegative function $v_K\in C_K^{\rm lip}$ such that $v_K \ge 1$ on $K$. 
Then for any $v\in C_K^{\rm lip}$ with ${\rm supp}[v]\subset K$, we see 
$$
|I(v)|\le \| v\|_{\infty} |I(v_K)|, 
$$
since $\|v\|_{\infty}v_K - v \in C_K^{\rm lip}$ and 
$\|v\|_{\infty}v_K - v \ge 0$. Note that any function $w\in C_K^{\rm lip}$ can be 
approximated  uniformly on $\mathbf{R}^d$ by functions $v\in C_K$ with 
${\rm supp}[v] \subset {\rm supp}[w]$. So, from the above inequality, 
$I$ can be extended uniquely to a positive linear functional on $C_K(\mathbf{R}^d)$.
Hence there exists a positive Radon measure $\mu$ on $\mathbf{R}^d$ so that 
(\ref{potential}) holds.
If the Dirichlet form $(\form,\dom)$ is transient, and, in addition,
$f$ is in $\dom_e$, then $f$ is excessive and the measure $\mu$ is smooth 
with respect to the form $\form$.

\item[{\sf (2)}] (Communicated by Professor Masatoshi Fukushima). Assume the 
Dirichlet form $(\form, \dom)$ is transient. Due to Theorem 1.5.3 \cite{FOT94}, 
the pair $(\dom_e, \form)$ is then a Hilbert space. 
Suppose a function $u\in\dom_e$ is $\form$-superharmonic in the following 
sense:
\[\form(u,v)\ge 0 \quad\]
for every $v \in \dom_e$ with $v\ge 0$.
Then we see $u\ge0$. In fact, since every normal contraction operates 
on $(\form, \dom_e)$, it follows that 
\[
\begin{array}{lcl}
\form(u,u) &\ge& \form(|u|,|u|)=\form\bigl((|u|-u)+u,(|u|-u)+u\bigr)  \\
&=& \form(u,u) +2\form(u, |u|-u) +\form(|u|-u,|u|-u) \\ 
&\ge& \form(u,u) +\form(|u|-u,|u|-u). \\ 
\end{array}
\]
This imples that $\form(|u|-u,|u|-u)=0$, whence $u=|u|\ge 0$.
\end{itemize}
\label{remark:transient}
\end{remark}


\section{Integral Derivation Properties} \label{Section;IDP}
We will prove various types of $L^p$-Liouville properties in the next 
section. The proof is based on a derivation property for $\Gamma$. 
The main purpose of this section is to establish its integral version for 
various classes of functions.

\begin{prop}[Derivation property of integral type]
Let $1\le p<\infty$ and $q$ be its conjugate; $1/p+1/q=1$. 
If $M_{\alpha}<\infty$ with some $0<\alpha\le 2$, then
\begin{align} 
&\int \Gamma(f,g\cdot h)(x) \, m(dx) \\
= &\int g(x) \Gamma(f,h)(x)\, m(dx) + 
\int h(x) \Gamma(f, g)(x)\, m(dx)
\label{eqn:derivation}
\end{align}
for all $f\in \dom_{{\rm loc},p} $, $g\in \dom_{{\rm loc},q}$, and 
$h\in C_K^{\rm lip}$.
\label{prop:deri}
\end{prop}

\begin{proof}
Let $f\in \dom_{{\rm loc},p}$, $g\in \dom_{{\rm loc},q}$ and 
$h\in C_K^{\rm lip}$. 
Recall that $\dom_{{\rm loc},p} =\dom_{\rm loc}\cap L^p$.  
It suffices to show that each of the integrals in (\ref{eqn:derivation}) 
converges. In fact,
\begin{align*}
& \dis{\int \Gamma(f,g\cdot h)(x) \, m(dx) }  \\ 
=&\dis{ \iint_{x\not=y} \bigl(f(x)-f(y)\bigr)\bigl(g(x)h(x)-g(y)h(y)\bigr)\mu(x,dy)m(dx) } \\ 
=&\dis{  \iint_{x\not=y} \bigl(f(x)-f(y)\bigr)} \\
& \times \dis{\bigl(g(x)\bigl(h(x)-h(y)\bigr)} 
\dis{ +\bigl(g(x)-g(y)\bigr)h(y)\bigr) \mu(x,dy)m(dx)} \\ 
= &\dis{ \iint_{x\not=y} g(x)\bigl(f(x)-f(y)\bigr)\bigl(h(x)-h(y)\bigr)
\mu(x,dy)m(dx)} \\ 
&\dis{+ \iint_{x\not=y} h(y) \bigl(f(x)-f(y)\bigr) \bigl(g(x)-g(y)\bigr) \mu(x,dy)m(dx).}
\end{align*}
Using the symmetry $\mu(x,dy)m(dx)=\mu(y,dx)m(dy)$ and a change of variables 
($x \leftrightarrow y$) in the second integral of the most right-hand side, 
we will 
obtain (\ref{eqn:derivation}).

Let $R-1 \ge r>0$ such that ${\rm supp}[h]\subset B(r)\subset B(R)$. 
Since $f\in \dom_{{\rm loc},p}$ and $g\cdot h\in \dom_{K,q}$, 
the left-hand side of (\ref{eqn:derivation}) converges absolutely by Lemma \ref{Lem;2.2}. 
Thus, we only consider the two integrals of the right-hand side. 
Take $f_R,g_R \in \dom$ so that $f=f_R$ and $g=g_R$, $m$-a.e.\ on $B(R)$, respectively.

We estimate the first term as follows:
\begin{align*} 
&\dis{\int \bigl| g(x)\bigr|  \bigl|\Gamma(f,h)(x) \bigr| \, m(dx)} \\  
\le & \dis{ \iint_{y\not=x} \bigl| g(x)\bigr|  
\bigl|f(x)-f(y)\bigr| \bigl|h(x)-h(y)\bigr| \mu(x,dy)\, m(dx)  }\\ 
\le & \dis{ \iint_{B(R)\times B(R)\backslash D} |g(x)|  
 \bigl| f(x)-f(y)\bigr|  \bigl| h(x)-h(y)\bigr| \mu(x,dy)\, m(dx) } \\ 
 & + \dis{ \iint_{B(r)\times B(R)^c} |g(x)|  
\bigl|f(x)-f(y)\bigr|  \bigl|h(x)\bigr| \mu(x,dy)\, m(dx) } \\ 
 & + \dis{ \iint_{B(R)^c\times B(r)} |g(x)|  
\bigl|f(x)-f(y)\bigr|  \bigl|h(y)\bigr| \mu(x,dy)\, m(dx) } \\ 
=&: {\rm (I)} + {\rm (II)} + {\rm (III)}.
\end{align*}
Here $D$ is the diagonal set $\{(x,x) : x\in {\mathbf{R}}^d \}$.
By making use of the Schwarz inequality and the fact that $h \in 
C_K^{\rm lip}$, 
\begin{align*}
{\rm (I)}  \le &
\dis{ \left( \iint_{B(R)\times B(R)\backslash D} |g_R(x)|^2 
\bigl| h(x)-h(y)\bigr|^2 \mu(x,dy)\, m(dx) \right)^{1/2} } \\ 
& \quad \times \dis{ \left( \iint_{B(R)\times B(R)\backslash D} 
   \bigl| f_R(x)-f_R(y)\bigr|^2 \mu(x,dy)\, m(dx)\right)^{1/2} } \\
\le &\dis{ \left(c_h \iint_{B(R)\times B(R)\backslash D} |g_R(x)|^2 
   \bigl(1 \wedge \bigl|x- y|^2 \bigr) \mu(x,dy)\, m(dx)\right)^{1/2} } \\
& \quad \times \dis{ \left(\iint_{B(R)\times B(R)\backslash D} 
   \bigl| f_R(x)-f_R(y)\bigr|^2 \mu(x,dy)\, m(dx) \right)^{1/2} } \\
&\le \dis{ \left( c_h M_{\alpha} \int_{B(R)} |g_R(x)|^2 m(dx) \right)^{1/2}
 \sqrt{ \form(f_R,f_R)} } \\
&\le \dis{ \sqrt{c_h M_{\alpha}} \, ||g_R||_{L^2}  \sqrt{ \form(f_R,f_R)}<\infty. }
\end{align*}
Here $c_h$ is the Lipschitz constant of $h$. 

In order to estimate (II) and (III), we establish the following inequality:
\begin{equation}
\int_K \int_{|x-y|\ge 1} \left(\bigl|\varphi(x)|^p +\bigl|\varphi(y)|^p\right)
\mu(x,dy)m(dx) \le 2M_{\alpha} \|\varphi\|_{L^p}^p
\label{est-1}
\end{equation}
 for every $\varphi \in L^p$ and any compact set $K\subset {\mathbf{R}}^d$.
In fact, by using the Fubini theorem and the symmety of $\mu(x,dy)m(dx)$, 
\begin{align*}
& \int_K \int_{|x-y|\ge 1} \left(\bigl|\varphi(x)|^p + 
  \bigl|\varphi(y)|^p\right) \mu(x,dy)m(dx)  \\ 
  = &
\int_K \bigl|\varphi(x)|^p \int_{|x-y|\ge 1} \mu(x,dy)m(dx) 
  + \int_K \int_{|x-y|\ge 1} \bigl|\varphi(y)|^p\mu(x,dy)m(dx)  \\ 
\le &  M_{\alpha} \int_K \bigl|\varphi(x)|^p m(dx) 
  + \int_K \int_{|x-y|\ge 1} \bigl|\varphi(y)|^p\mu(y,dx)m(dy)   \\
= & M_{\alpha} \int_K \bigl|\varphi(x)|^p m(dx) 
  + \int_{{\mathbf{R}}^d} \bigl|\varphi(y)|^p \int_{K \cap |x-y|\ge 1} \mu(y,dx)m(dy)  \\ 
\le &  2 M_{\alpha} \|\varphi\|_{L^p}^p. 
\end{align*}

We now estimate (II). By the H\"older inequality and the fact $R-r\ge 1$, 
\begin{align*}
{\rm (II)} \le & 
\dis{ \left( \iint_{B(r)\times B(R)^c}  
\bigl|g(x)\bigr|^q \bigl|h(x)\bigr|^q
 \mu(x,dy)\, m(dx)\right)^{1/q} } \\  
   & \dis{  \times \left( \iint_{B(r)\times B(R)^c}  \bigl|f(x)-f(y)\bigr|^p \mu(x,dy)\, 
 m(dx)\right)^{1/p}  } \\ 
\le& \dis{ \left(\|h\|_{\infty}^q \int_{B(r)} \bigl|g(x)\bigr|^q 
\int_{B(R)^c} \mu(x,dy)m(dx) \right)^{1/q} } \\ 
& \dis{ \times \left(2^{p-1} \int_{B(r)} \int_{|x-y| \ge 1}
\left(|f(x)|^p +|f(y)|^p \right) \mu(x,dy)m(dx) \right)^{1/p}  } \\ 
\le& \dis{2 M_{\alpha} \|h\|_{\infty}  \|g\|_{L^q}  \|f\|_{L^p}, }
\end{align*}
where we used (\ref{est-1}) in the last inequality. 

Moreover, using H\"older inequality again, 
\begin{align*}
{\rm (III)} \le & \dis{ 
\left( \iint_{B(R)^c \times B(r)} \bigl|g(x)\bigr|^q  
\bigl| h (y)\bigr|^q \mu(x,dy)m(dx) \right)^{1/q} } \\  
& \dis{ \times \left( \iint_{B(R)^c \times B(r)} \bigl|f(x)-f(y) |^p 
\mu(x,dy)m(dx) \right)^{1/p}. } 
\end{align*}
The right-hand side can be estimated by 
$2 M_{\alpha} \|h\|_{\infty}  \|g\|_{L^q}  \|f\|_{L^p}$
as in (II). Hence we conclude that 
the first term of the right-hand side of (\ref{prop:deri}) 
converges absolutely.

\medskip
We proceed to estimate the second term of the right-hand side of 
(\ref{prop:deri}). Since the support of $h$ is contained in $B(r)\subset B(R)$, 
\begin{align*}
      &\dis{\int \bigl| h(x)\bigr|  \bigl|\Gamma(f,g)(x) \bigr| \, m(dx)} \\  
\le &\dis{ \iint_{y\not=x} \bigl| h(x)\bigr|  
\bigl|f(x)-f(y)\bigr| \bigl|g(x)-g(y)\bigr| \mu(x,dy)\, m(dx)  }\\ 
=&\dis{ \iint_{B(r)\times B(R) \backslash D} \bigl| h(x)\bigr|  
\bigl|f(x)-f(y)\bigr| \bigl|g(x)-g(y)\bigr| \mu(x,dy)\, m(dx)  }\\ 
&+\dis{ \iint_{B(r)\times B(R)^c} \bigl| h(x)\bigr|  
\bigl|f(x)-f(y)\bigr| \bigl|g(x)-g(y)\bigr| \mu(x,dy)\, m(dx)  }\\ 
=&: {\rm (I)} + {\rm (II)}. \\
\end{align*}
By using the Schwarz inequality, 
\begin{align*}
 {\rm (I)} \le& \dis{ \|h\|_{\infty} \iint_{B(R)\times B(R) \setminus D} 
  \bigl|f_R(x)-f_R(y)\bigr| \bigl|g_R(x)-g_R(y)\bigr| \mu(x,dy)\, m(dx)  }\\ 
  \le& \dis{ \|h\|_{\infty} \sqrt{ \iint_{B(R)\times B(R) \setminus D} 
     \bigl|f_R(x)-f_R(y) \bigr|^2 \mu(x,dy)\, m(dx)}  }\\ 
   & \times \dis{ \sqrt{ \iint_{B(R)\times B(R) \setminus D} 
   \bigl|g_R(x)-g_R(y) \bigr|^2 \mu(x,dy)\, m(dx)}  }\\ 
 \le& \dis{ \|h\|_{\infty}  \sqrt{\form(f_R,f_R)} \sqrt{\form(g_R,g_R)}.}
\end{align*}
By the H\"older inequatliy, 
\begin{align*}
 {\rm (II)} &\le  \dis{ \|h\|_{\infty}
\iint_{B(r)\times B(R)^c} 
\bigl|f(x)-f(y)\bigr| \bigl|g(x)-g(y)\bigr| \mu(x,dy)\, m(dx)  }\\ &\le \dis{ \|h\|_{\infty}
\left(\iint_{B(r)\times B(R)^c} \bigl|f(x)-f(y)\bigr|^p 
\mu(x,dy)\, m(dx)  \right)^{1/p} } \\ &   \times \dis{ 
\left(\iint_{B(r)\times B(R)^c} \bigl|g(x)-g(y)\bigr|^q 
\mu(x,dy)\, m(dx)  \right)^{1/q} } \\ &\le \dis{ \|h\|_{\infty}
\left(2^{p-1} \iint_{B(r)\times B(R)^c} \left( \bigl|f(x)\bigr|^p + 
\bigl|f(y) \bigr|^p\right) \mu(x,dy)\, m(dx)  \right)^{1/p} } \\ &  \times \dis{ 
\left(2^{q-1} \iint_{B(r)\times B(R)^c} \left( \bigl|g(x)\bigr|^q + 
\bigl|g(y) \bigr|^q \right) \mu(x,dy)\, m(dx)  \right)^{1/q}, }
\end{align*}
where the last term is estimated by 
\[4 M_{\alpha} \|h\|_{\infty} \|f\|_{L^p} \|g\|_{L^q}\]
using 
 (\ref{est-1}). 
Thus the second term of the right-hand side of 
(\ref{prop:deri}) also converges absolutely. 
Now the proof completes.
\end{proof}

\begin{remark} 
\begin{itemize}
\item[\sf (1)] Carre du champ operators are originally defined by 
Roth \cite{R74, R76} and then extensively studied by Bouleau and Hircsch 
in their book \cite{BH91} mainly in the local operetors case (see also 
\cite{B94}, \cite{AMSh94}).

\item[\sf (2)] Carlen, Kusuoka and Stroock \cite{CKS87} have 
already obtained the derivation property for integral type
for a general Dirichlet form for the bounded functions. 
Namely, 
\begin{align*}
\form(f,gh)&=\int \Gamma(f,g\cdot h)(x)\, m(dx) \\ &=\int \tilde{g}(x)\Gamma(f,h)(x)\, m(dx) 
+ \int \tilde{h}(x) \Gamma(f,g)\, m(dx)
\end{align*}
for all $f,g,h \in \dom_b=\form \cap L^\infty$,
where $\tilde{u}$ is a quasi-continuous modification of $u$. 
Here we give a simple proof of this fact. 
Recall 
equation (see \cite[(3.2.14)]{FOT94} or 
\cite[Lemma 3.1.]{K02}):
\[
2 \int \tilde{f}(x)\Gamma(u,u)\, m(dx)=2\form(uf,u)-\form(u^2,f), \quad 
f,u \in \dom_b.
\]
Inserting $u=g+h$ and $u=g-h$ respectively into $g,h\in \dom_b$, 
and then subtracting the inserted equations, 
\begin{equation}
2\int \tilde{f}\Gamma(g,h)\, m(dx)=\form(gf,h)+\form(hf,g)-\form(gh,f).
\label{eqn:deri}
\end{equation}
Then make change the functions $f \leftrightarrow h$ in the above, 
\begin{equation}
2\int \tilde{h}\Gamma(g,f)\, m(dx)=\form(gh,f)+\form(fh,g)-\form(gf,h).
\label{eqn:deri2}
\end{equation}
Adding each sides of (\ref{eqn:deri}) and (\ref{eqn:deri2}) and 
using the symmetry, we have 
\[
\int \tilde{f}\Gamma(h,g)\, m(dx) +\int \tilde{h}\Gamma(f,g)\, m(dx)
=\form(fh,g)=\int \Gamma(fh,g)\, m(dx).
\]
This is the desired equation.
We observe that the assumption such that the functions are bounded is essential,
while, our derivation property allows the functions to be unbounded.
\end{itemize}
\end{remark}

We can extend Proposition \ref{prop:deri} in two cases.
The first case is the ``truncated jump kernel":
\begin{cor}\label{Cor;DT} 
Let $\mu(x,dy)$ be a kernel such that the support is included in $B_x(R)$ 
with some $R>0$ for any $x\in {\mathbf{R}}^d$. 
Let $1\le p<\infty$ and $q$ be its conjugate. 
If $M_{\alpha} < \infty$ for some $0 < \alpha \le 2$,
then 
\[
\int \Gamma(f,g\cdot h)(x)\, m(dx)
=\int g(x) \Gamma(f,h)(x)\, m(dx) + 
\int h(x) \Gamma(f, g)(x)\, m(dx)
\]
for all $f \in \dom_{\rm loc}\cap L^p_{\rm loc}$, 
$g \in \dom_{\rm loc}\cap L^q_{\rm loc}$, and $h \in C_K^{\rm lip}$.
\end{cor}

The second case is for H\"older continuous functions:

\begin{cor}  Let $1 \le p<\infty$ and $q$ be the conjugate.
Let  $0<\beta_1, \beta_2\le 1$ such that $\beta_1+\beta_2\ge \alpha$. 
If $M_{\alpha}< \infty$ for some $0<\alpha\le 2$,
then 
\[
\int \Gamma(f,g\cdot h)(x) \, m(dx)=\int g(x) \Gamma(f,h)(x)\, m(dx) + 
\int h(x) \Gamma(f, g)(x)\, m(dx)\]
for all $f\in C_{\rm loc}^{\beta_1} \cap L^p$,  
$g\in C_{\rm loc}^{\beta_2} \cap L^q$, and $h\in C_K^{\rm lip}$.
\label{cor:deri}
\end{cor}
Here $C_{\rm loc}^{\beta}=C_{\rm loc}^{\beta}({\mathbf{R}}^d)$ is the set of all 
locally $\beta$-H\"older continuous functions $f$ for $0<\beta\le 1$; 
namely, $f\in C_{\rm loc}^{\beta}$ if and only if
\[
\|f\|_{\beta(K)} = \sup_{x,y\in K} \frac{|f(x)-f(y)|}{|x-y|^{\beta}}<\infty
\] 
for any compact set $K \subset {\mathbf{R}}^d$.
We also use the following classes of H\"older continuous 
functions:
\begin{itemize}
\item[$\bullet$] 
$\dis{C^{\beta}=C^{\beta}({\mathbf{R}}^d) \ni f \quad \Longleftrightarrow  \quad 
\|f\|_{\beta} = \sup_{|x-y|\le 1} 
\frac{|f(x)-f(y)|}{|x-y|^{\beta}}<\infty.}$
\item[$\bullet$] $\dis{C_K^{\beta}=C_{\rm loc}^{\beta} \cap C_K({\mathbf{R}}^d).}$
\end{itemize}
Clearly the following inclusions hold:
\[
C^\beta_K \subset \tilde{C}^\beta \subset C^\beta \subset C^\beta_{\rm loc},
\]
where $\tilde{C}^\beta$ is the set of uniformly $\beta$-H\"older continuous 
functions. Note that $C^{\beta_1}_{\rm loc}$ is a proper subspace of 
$C^{\beta_2}_{\rm loc}$ if $\beta_1 > \beta_2$. Of course when $\beta=1$, 
$C^{\beta}_{\rm loc}$ is nothing but the space of (local) Lipschitz 
continuous functions, which will be denoted by $C^{\rm lip}_{\rm loc}$.


\section{$L^p$-Liouville Property for $2\le p<\infty$}
\label{LP}
In this section we prove our main theorem (Theorems \ref{MT1}) for the case 
$2\le p<\infty$. Before starting the proof, we recall 
a simple lemma:
\begin{lemma}  If $2\le p<\infty$, then we have
\[
|x-y|^p \le (x-y)(x^{p-1}-y^{p-1}) \mbox{ for every $x,y \in {\mathbf R}$ with $x,y\ge 0.$}\]
\label{lemma:elementary}
\end{lemma}

\begin{proof}Since the inequality holds as the equality when $p=2$,
we show it only for $2<p<\infty$. Assume $x>y>0$. 
Then
\[
(x-y)(x^{p-1}-y^{p-1})-(x-y)^p=y^{p-1}(x-y)\Bigg\{ 
\left(\frac{x}{y}\right)^{p-1}-1 -\left(\frac xy -1\right)^{p-1}\Bigg\}.
\]
Thus, it is enough to show: 
\[
t^p-1- (t-1)^p \ge 0 \mbox{ for every $ t\ge 1$ and $1<p<\infty$.}
\]
Setting $f(t)$ as the left-hand side of the above, it follows:
\[
f'(t)=p t^{p-1}-p(t-1)^{p-1}
=p t^{p-1} \left(1-(1-1/t)^{p-1}\right)>0 \mbox{ for $t>1$.}
\]
Hence, $f(t)$ is strictly increasing on $[1,\infty)$ and $f(1)=0$,
and thus, $f(t)\ge 0$ for $t\ge 1$.
\end{proof}

Now we prove:  
\begin{theorem}[$L^p$-Liouville property for $2\le p<\infty$] \label{MT1}
Let $2\le p< \infty$ and $q$ be its conjugate.
If $M_{q}<\infty$, 
then any nonnegative $\form$-subharmonic function 
$f \in C \cap \dom_{\rm loc} \cap L^p$ is  constant.
\end{theorem}
\begin{proof}
Let $f \in C({\mathbf{R}}^d) \cap \dom_{{\rm loc},p} (\subset {\mathcal H})$ be a
nonnegative $\form$-subharmonic function. Note that $1<q\le 2$ 
and $M_2 \le M_{q}<\infty$, since $2\le p<\infty$. Recall that 
${\mathcal C}=C_K({\mathbf{R}}^d) \cap \dom \subset \dom_{K,q}$ and note 
that $f^{p-1} \in \dom_{{\rm loc},q} \cap C({\mathbf{R}}^d)$.  
Take a sequence of cut-off functions 
$\{ \chi_n\}_{n \in {\mathbf{N}}} \subset C_K^{\rm lip}$ satisfying the following conditions:
\[0 \le \chi_n \nearrow 1 \ \ {\rm as} \ n\to \infty,\quad
|\chi_n(x)-\chi_n(y)| \le |x-y| \ \ {\rm for} \ \ x,y \in {\mathbf{R}}^d.\]
Since $2\le p<\infty$, the function $\chi^2_n f^{p-1}$ is 
non-negative and belongs to ${\mathcal C}$. 
So, we may apply the integral derivation property (Proposition \ref{prop:deri}),
and it follows that
\begin{align*}
0 &\le -\form(f,\chi_n^2 f^{p-1}) \\  
&= \dis{-\int \Gamma(f, \chi_n^2  f^{p-1})(x)\, m(dx)  } \\  
&= -\dis{ \int \chi^2_n(x) \Gamma(f,f^{p-1})\, m(dx) - \int f^{p-1} 
\Gamma(f, \chi^2_n)\, m(dx)}
\end{align*}
for any integer $n>0$. 
Moreover,  by Lemma \ref{lemma:elementary},
\begin{align*}
&\int \chi^2_n(x) \Gamma(f,f^{p-1})\, m(dx) \\ 
=&\iint_{y\not=x} \chi^2_n(x) \left( f(x)-f(y) \right)
\bigl(f^{p-1}(x)-f^{p-1}(y)\bigr) \mu(x,dy)m(dx) \\  
\ge& \iint_{x \neq y} \chi^2_n (x) |f(x)-f(y)|^p \, \mu(x,dy)m(dx) \ge 0.
\end{align*}
Thus,\begin{equation} \label{Eq;3.1}
\dis{0 \le \int \chi^2_n (x) \Gamma(f,f^{p-1})(x) \, m(dx) \le 
\Biggl| \int f^{p-1}(x) \Gamma(f, \chi^2_n)(x) \, m(dx) \Biggr|. }
\end{equation}
We estimate the right-hand side of (\ref{Eq;3.1}). 
Since the assumption $p\ge 2$ implies $2/p \le 1$ and
$0\le \chi_n(x) \le \chi_n(x)^{2/p} \le 1$, we have that 
\begin{align*}
&\dis{ \Bigl| \int f^{p-1}(x) \Gamma(f, \chi_n^2)(x) \, m(dx) \Bigr| } \\ 
\le& \dis{ \iint_{x\not=y} f^{p-1}(x)\bigl| f(x)-f(y)\bigr| 
   \bigl| \chi_n^2(x)-\chi_n^2(y)\bigr| \mu(x,dy)m(dx) } \\  
\le &\dis{  \iint_{x\not=y} f^{p-1}(x) \chi_n(x) \bigl|f(x)-f(y)\bigr| 
   \bigl|\chi_n(x)-\chi_n(y)\bigr| \mu(x,dy)m(dx) } \\  
&\dis{ +\iint_{x\not=y} f^{p-1}(x) \chi_n(y)\bigl|f(x)-f(y)\bigr| 
   \bigl|\chi_n(x)-\chi_n(y)\bigr| \mu(x,dy)m(dx) } \\  
\le &\dis{  \iint_{x\not=y} \bigl(f^{p-1}(x)   |\chi_n(x)-\chi_n(y)| \bigr) 
   \chi^{2/p}_n(x)  |f(x)-f(y)|  \mu(x,dy)m(dx) } \\  
& \dis{ +  \iint_{x\not=y} \bigl(f^{p-1}(x)   |\chi_n(x)-\chi_n(y)| \bigr)
   \chi^{2/p}_n(y)  |f(x)-f(y)| \mu(x,dy)m(dx) } \\  
\le &\dis{  \left( \iint_{x\not=y} \bigl(f^{p-1}(x)   
  |\chi_n(x)-\chi_n(y)| \bigr)^q  \mu(x,dy)m(dx) \right)^{1/q} } \\ 
&\dis{\times \left( \iint_{x\not=y}\bigl(\chi^{2/p}_n(x) 
   |f(x)-f(y)| \bigr)^p \mu(x,dy)m(dx) \right)^{1/p}  } \\ 
&\dis{+  \left( \iint_{x\not=y} \bigl(f^{p-1}(x)   
  |\chi_n(x)-\chi_n(y)| \bigr)^q  \mu(x,dy)m(dx) \right)^{1/q}  } \\   
&\dis{\times \left( \iint_{x\not=y}\bigl(\chi^{2/p}_n(y) 
   |f(x)-f(y)| \bigr)^p \mu(x,dy)m(dx) \right)^{1/p}  } \\ 
\le &\dis{ 2  \left( \iint_{x\not=y}  f^p(x) {}  
|\chi_n(x){-}\chi_n(y)|^q  \mu(x,dy)m(dx) \right)^{1/q}  } \\  
&\dis{\times \left( \iint_{x\not=y} \chi^2_n(x) {} 
|f(x){-}f(y)|^p \mu(x,dy)m(dx) \right)^{1/p},  } 
\end{align*}
where, 
we used the H\"older inequality, changing the variables 
$x\leftrightarrow y$, and the symmetry of $\mu(x,dy)m(dx)$
in the second integral of the most right-hand side. 

Applying Lemma \ref{lemma:elementary}, 
the second integral of the most right-hand side is dominated by 
\begin{align*}
&\iint_{x\not=y} \chi_n(x)^2 \bigl(f(x)-f(y)\bigr)
\bigl(f^{p-1}(x)-f^{p-1}(y)\bigr) \mu(x,dy)m(dx) \\ = & \int \chi_n(x)^2 
\Gamma(f,f^{p-1})\, m(dx).
\end{align*}
Thus
\begin{align*}
&\dis{\Bigl|  \int f^{p-1}(x) \Gamma(f, \chi_n^2)(x)\, m(dx) \Bigr| } \\
\le&2  \dis{ \left( \iint_{x\not=y} f^p(x) 
\bigl|\chi_n(x)-\chi_n(y)\bigr|^q \mu(x,dy)m(dx) \right)^{1/q}  } \\
& \dis{ \times \left( \int \chi_n(x)^2 \Gamma(f,f^{p-1})\, m(dx) \right)^{1/p}. } 
\end{align*}
Hence, by (\ref{Eq;3.1}) we have the following:
\begin{align*}
&\int \chi^2_n(x) \Gamma(f,f^{p-1})\, m(dx) \\ 
\le 2  & \left( \iint_{x\not=y} f^p(x) \bigl|\chi_n(x)-\chi_n(y)\bigr|^q 
\mu(x,dy)m(dx) \right)^{1/q} \\ 
\times &\left( \iint \chi^2_n(x) \Gamma(f,f^{p-1})\, 
m(dx) \right)^{1/p},
\end{align*}
and thus
\[
\int \chi^2_n(x) \Gamma(f,f^{p-1})\, m(dx) 
\le 2^q \iint_{x\not=y} \hspace{-10pt} 
f^p(x) \bigl|\chi_n(x)-\chi_n(y)\bigr|^q 
\mu(x,dy)m(dx). 
\]
Since $\chi_n \nearrow1$, $f^p(x) 
\bigl|\chi_n(x)-\chi_n(y)\bigr|^q$ tends to $0$ for 
$\mu(x,dy)m(dx)$
-a.e.\ $(x,y) \in {\mathbf{R}}^d\times {\mathbf{R}}^d$ as $n\to\infty$. Moreover 
\[
f^p(x) \bigl|\chi_n(x)-\chi_n(y)\bigr|^q \le f^p(x) \left(
1\wedge |x-y|^q \right)
\]
where the right-hand side is integrable with respect to 
$\mu(x,dy)m(dx)$ 
due to the assumption $M_{q} < \infty$ together with the fact that $f\in L^p$. 
Applying the Lebesgue theorem and the Fatou's lemma, it follows that
\begin{align*}
&\dis{\int \Gamma(f,f^{p-1})\, m(dx)} \\  
\le &\dis{ \liminf_{n\to\infty} \int \chi_n(x)^2 \Gamma(f,f^{p-1})\, m(dx)}  \\ 
\le &\dis{ 4 \lim_{n\to\infty} \iint_{x\not=y} 
f(x)^p \bigl| \chi_n(x)-\chi_n(y)\bigr|^q \mu(x,dy)m(dx) =0,} 
\end{align*}
where the most left-hand side is non-negative by Lemma 
\ref{lemma:elementary}. Hence 
\[\dis{
\Gamma(f,f^{p-1})(x) =\int_{y\not=x}
\bigl(f(x)-f(y)\bigr)\bigl(f^{p-1}(x)-f^{p-1}(y)\bigr) \mu(x,dy)= 0}\]
for $m\mbox{-a.e.}\ x\in{\mathbf{R}}^d.$
Finally, since $\bigl(f(x)-f(y)\bigr)\bigl(f^{p-1}(x)-f^{p-1}(y)\bigr) \ge 0$ 
for every $(x,y) \in {\mathbf{R}}^d \times {\mathbf{R}}^d$, we arrive at the conclusion
because $M_q<\infty$ and $f$ is continuous.
\end{proof}

\section{$L^p$-Liouville Property for $1<p<2$}\label{Section;LPII}

Let $1<p<2$ and assume $M_1<\infty$. Under this setting, we will show two 
sorts of $L^p$-Liouville properties; first with truncated jump kernel,
secondly, for H\"older continuous subharmonic functions.

We can not apply directly the proof of the Liouville property with $2\le p<\infty$.
Because the function $f^{p-1}$ may not belong 
to $\dom_{\rm loc}$ if $1<p<2$; accordingly, we can not use 
$\chi_n^2 f^{p-1}$ as a test function. 
In the diffusion (local) case, Sturm \cite{S94} considered
(the original idea is due to Yau \cite{Y76})
the function $(f+\varepsilon)^{p-1}$, which belongs to $\dom_{\rm loc}$, 
instead of $f^{p-1}$ with $\varepsilon>0$, and followed the line of the proof with 
$2 \le p <\infty$. Then he completed the proof by letting $\varepsilon \to 0$. 

The crucial fact which lets him carry out this proof is that a local form 
$\form$ can be extended to the space $\dom_{\rm loc} \times 
\bigl(C_K({\mathbf{R}}^d)\cap \dom\bigr)$.
A non-local form $\form$ is in general 
not extendable to this space. 
Nevertheless, we will show the Liouville property by 
modifying the jump kernel or the definition of a harmonic function.

\subsection{Truncated Jump Kernel}

In this subsection, we will show the $L^p$-Liouville property 
for the kernel $\mu$ of the following type:
\begin{equation}
\mu(x,dy)={\bf 1}_{|x-y|\le R} \, \mu(x,dy)
\label{kernel}
\end{equation}
with some $R>0$ for every $x\in{\mathbf{R}}^d$. Namely, the measure $\mu(x,dy)$ 
is supported in  $\{ y : |x-y|\le R\}$ for each $x\in{\mathbf{R}}^d$. 
Note that the truncate constant $R>0$ is not essential in the following 
arguments and we let $R=1$ for the sake of simplicity.

We first show that the form $\form$ can be extended to the space
$\dom_{\rm loc} \times \bigl(C_K \cap \dom\bigr)$ or 
$\bigl(C_K \cap \dom \bigr) \times \dom_{\rm loc}$. 

\begin{lemma} (c.f. Lemma \ref{lemma:integral})  Let $\mu(x,dy)$ be 
a kernel that satisfies Condition (\ref{kernel}).
If $M_{\alpha}<\infty$ with some $0<\alpha\le 2$, then the following
integral converges:
\[
\iint_{x\not=y} \bigl|f(x)-f(y)\bigr| \bigr|g(x)-g(y)\bigr| 
\mu(x,dy)m(dx)<\infty
\]
for any $f\in \dom_{\rm loc}$ and $g\in C_K \cap \dom$.
Namely,
\[
\form(f,g)=\form(g,f)=\iint_{x\not=y} \bigl(f(x)-f(y)\bigr) 
\bigr(g(x)-g(y)\bigr) \mu(x,dy)m(dx)
\]
makes sense for such a pair $f$ and $g$. 
\label{lemma:integral-2}
\end{lemma}
\begin{proof}Let $f\in \dom_{\rm loc}$ and $g\in C_K({\mathbf{R}}^d)\cap \dom$. 
Let $R,r>0$ such that ${\rm supp}[g]\subset B(r) \subset B(R)$ with $R-r\ge 1$.
Take $f_R\in \dom$ such that $f=f_R \ m$-a.e.\ on $B(R+1)$.
Since the support of the measure $\mu(x,y)m(dx)$ is included in $\{|x-y|\le 1\}$,
it follows that
\begin{align*}
 &\dis{
\iint_{x\not=y} \bigl|f(x)-f(y)\bigr| \bigr|g(x)-g(y)\bigr| \mu(x,dy)m(dx)} \\ 
=&\dis{\int_{B(R)} \int_{0<|x-y|\le 1} \bigl|f_R(x)-f_R(y)\bigr| 
\bigr|g(x)-g(y)\bigr| \mu(x,dy)m(dx)} \\ 
\le &\dis{\left(\int_{B(R)} \int_{0<|x-y|\le 1}\bigl(f_R(x)-f_R(y)\bigr)^2 
\mu(x,dy)m(dx) \right)^{1/2}  }  \\ 
& \quad \times \dis{\left(\int_{B(R)} \int_{0<|x-y|\le 1}\bigl(g(x)-g(y)\bigr)^2
\mu(x,dy)m(dx) \right)^{1/2}  } \\ 
\le & \dis{ \left( \iint_{x\not=y}\bigl(f_R(x)-f_R(y)\bigr)^2 \mu(x,dy)m(dx) \right)^{1/2}
 \left(\iint_{x\not=y}\bigl(g(x)-g(y)\bigr)^2\mu(x,dy)m(dx)\right)^{1/2} } \\ 
= &\sqrt{\form(f_R,f_R)}\sqrt{\form(g,g)} <\infty.
\end{align*}
%
%
%
%
%
\end{proof}

Now we prove the $L^p$-Liouville property with
$1<p<2$ for the truncated jump kernel.

\begin{theorem} Let $1<p<2$ and $\mu(x,dy)$ be a jump kernel satisfying (\ref{kernel}).
If $M_1<\infty$, then any 
nonnegative $\form$-subharmonic function $f \in C \cap \dom_{{\rm loc},p} 
\bigl(=C \cap \dom_{\rm loc}\cap L^p \bigr)$ is constant.
\label{theorem:truncate}
\end{theorem}
\begin{proof}
Let $f \in C \cap \dom_{{\rm loc},p} (\subset {\mathcal H})$ be a
nonnegative and $\form$-subharmonic function. Similar to the proof of Theorem 
\ref{MT1}, take a sequence of cut-off functions  
$\{ \chi_n\}_{n=1}^{\infty} \subset C_K^{\rm lip}({\mathbf{R}}^d)$ such that
\[
0 \le \chi_n \nearrow 1 \ \ {\rm as} \ n\to \infty,\quad
|\chi_n(x)-\chi_n(y)| \le |x-y| \ \ {\rm for} \ \ x,y \in {\mathbf{R}}^d.
\]
We can easily show that $(f+\vareps)^{p-1}  \in \dom_{\rm loc}$ for any 
$\vareps>0$ by the inequality:
\[
\Bigl|\bigl(t+\vareps\bigr)^{p-1} -\bigl(s+\vareps\bigr)^{p-1}\Bigr| 
\le (p-1) \, \vareps^{p-2} \bigl|t-s\bigr| \mbox{ for  any $ t,s \ge 0$}.
\]
Take $R_n,r_n>0$ for each integer $n>0$ such that 
${\rm supp}[\chi_n] \subset B(r_n) \subset B(R_n)$ and $R_n-r_n\ge 1$. 
Set $\vareps_n= \bigl(n \cdot m(B(R_n))\bigr)^{-1/p}>0$. Then 
$\vareps_n \downarrow 0$ as $n\to\infty$. 
By making use of the derivation property (Corollary \ref{Cor;DT}), we have that
\begin{align*}
 0 \le & -\form(f,\chi_n^2 (f+\vareps_n)^{p-1}) \\
     = &\dis{-\int  \Gamma(f, \chi_n^2 (f+\vareps_n)^{p-1}) m(dx) }  \\  
= & \dis{-\int \chi_n^2(x) \Gamma(f, (f+\vareps_n)^{p-1}) m(dx) 
- \int (f(x)+\vareps_n)^{p-1} \Gamma(f, \chi_n^2 ) m(dx),  }
\end{align*}
hence, as in the case $2\le p<\infty$, we have:
\[
0\le \int \chi_n^2(x) \Gamma(f, (f+\vareps_n)^{p-1}) m(dx)
\le \Bigl| \int (f(x)+\vareps_n)^{p-1} \Gamma(f, \chi_n^2 ) m(dx) \Bigr|.
\]
We estimate the right-hand side as follows:
\begin{align*}
&\dis{ \Bigl| \int (f(x)+\vareps_n)^{p-1} \Gamma(f, \chi_n^2 ) m(dx) \Bigr| } \\  
\le &\dis{ \iint_{0<|x-y|\le 1} (f(x)+\vareps_n)^{p-1} 
   |f(x)-f(y)|  |\chi_n^2(x)-\chi_n^2(y) | \mu(x,dy)m(dx) } \\  
\le & \dis{ 2 \iint_{0<|x-y|\le 1}(f(x){+}\vareps_n)^{p-1}|\chi_n(x)-\chi_n(y) |^{1/q}  
   |f(x)-f(y)| } \\  
& \dis{ \times |\chi_n(x)-\chi_n(y) |^{1/p} \mu(x,dy)m(dx) } \\  
\le & \dis{2 \left(\iint_{0<|x-y|\le 1}  (f(x){+}\vareps_n)^{p}  
   |\chi_n(x)-\chi_n(y) |\mu(x,dy)m(dx) \right)^{1/q} } \\  
& \times \dis{\left( \iint_{0<|x-y|\le 1}  |f(x)-f(y)|^p 
|\chi_n(x)-\chi_n(y) |\mu(x,dy)m(dx) \right)^{1/p} } \\  
=&: 2 {\bf (I)_n}^{1/q}  {\bf (II)}_n^{1/p}.
\end{align*}
where we used the H\"older inequality in the last inequality. 

Since $R_n-r_n\ge 1$ and the support of $\chi_n$ is included in $B(r_n$), 
it follows that
\begin{align*}
{\rm (I)}_n =& \dis{ \int_{B(R_n)} \int_{0<|x-y|\le 1} 
(f(x){+}\vareps_n)^{p}  |\chi_n(x)-\chi_n(y) |\mu(x,dy)m(dx) } \\ 
\le& \dis{ 2^{p-1} \int_{B(R_n)} \int_{0<|x-y|\le 1} 
\left( f^p(x)+ \vareps_n^p \right)  |x-y|\mu(x,dy)m(dx) } \\ 
\le& \dis{ 2^{p-1} M_1 \left(\int_{B(R_n)} f^p(x) m(dx) 
+\vareps_n^p \mu(B(R_n)) \right) } \\ \le& \dis{ 2^{p-1} M_1 \bigr( \|f\|_{L^p}^p + 1 \bigr). }
\end{align*}
Hence, {\rm (I)}$_n$ are uniformly bounded in $n>0$. 
Next, we estimate {\rm (II)}$_n$. 
Applying the symmetry of 
the measure $\mu(x,dy)m(dx)$, 
\begin{equation}
\begin{array}{lcl} {\rm (II)}_n 
&\le& \dis{2^{p-1} \iint_{0<|x-y|\le 1}  
(f^p(x)+ f^p(y))|\chi_n(x)-\chi_n(y)| \mu(x,dy)m(dx) } \\
& \le & \dis{ 2^p \iint_{0<|x-y|\le 1} f^p(x) |\chi_n(x)-\chi_n(y)| 
 \mu(x,dy)m(dx). } \\ 
\end{array}
\label{truncate-conv}
\end{equation}
On the other hand,
\[f^p(x) |\chi_n(x)-\chi_n(y)|  \le f^p(x) |x- y|,\]
where the right-hand side is integrable with respect to $\mu(x,dy)m(dx)$. 
Since $\chi_n$ converges to $1$ as $n\to\infty$, we find that 
$f^p(x) |\chi_n(x)-\chi_n(y)|$ converges to $1$ for $\mu(x,dy)m(dx)$-a.e. 
$(x,y) \in\mathbf{R}^d\times\mathbf{R}^d$.  
So, by the dominated convergence theorem, we find the right-hand side of 
(\ref{truncate-conv}) tends to $0$ as $n\to \infty$. 
Thus by making use of Fatou's lemma, we have:
\begin{align*}
 \int \Gamma(f,f^{p-1})(x)m(dx) = & \int \lim_{n\to\infty} \chi_n(x)^2 
\Gamma(f,(f+\vareps_n)^{p-1})(x) m(dx) \\ 
\le &  \liminf_{n\to\infty} \int \chi_n(x)^2 
\Gamma(f,(f+\vareps_n)^{p-1})(x) m(dx)  \\ 
\le & \liminf_{n\to\infty} \Bigl| \int \bigl(f(x)+\vareps_n\bigr)^{p-1}
\Gamma(f,\chi_n^2)(x)m(dx)\Bigr| =0.  \\
\end{align*}
The rest of the proof is the same as the proof of Theorem \ref{MT1}.
\end{proof}

\subsection{H\"older continuous harmonic functions}

In this subsection, we will show the $L^p$-Liouvielle property
for 
H\"older continuous harmonic functions. 

\smallskip
As in the Lemma \ref{lemma:integral}, we can extend the quadratic form 
$\form$ to $\bigl(C^{\beta_1}_{\rm loc} \cap L^p({\mathbf{R}}^d;m)\bigr)\times 
C_K^{\beta_2}$ or \\
{$C_K^{\beta_2}\times \bigl(C^{\beta_1}_{\rm loc} \cap 
L^p({\mathbf{R}}^d;m)\bigr)$} with $\beta_1+\beta_2 \ge \alpha$. 

We 
now set 
\[
\mathcal{H}(\gamma) = 
\{ 
u \in L^0: \ \form(u,\psi) \ \mbox{makes sense for all} \ \psi \in 
C^\gamma_K \}
\]
for $0<\gamma\le \alpha$.  As in the previous case, it follows:
\[
\dom \subset \dom_e \subset {\mathcal H}(\gamma) \quad {\rm and} \ \ 
\bigcup_{1\le p<\infty} \left(C^{\beta}_{\rm loc} 
\cap L^p({\mathbf{R}}^d;m)\right) \subset {\mathcal H}(\gamma),
\]
where $\beta+\gamma\ge \alpha$. 
We define the harmonic functions in this setting as follows:
\begin{definition} 
We call a  measurable function $f$ defined on ${\mathbf{R}}^d$
\emph{$\form$-subharmonic} ({\it $\form$-superharmonic}, respectively) 
with exponent $\gamma$ if $f\in {\mathcal H}(\gamma)$ and 
\[
\form(f,\varphi) \le 0 \quad (\form(f,\varphi)\ge 0, \mbox{respectively})
\]
for any $\varphi \in C_K^{\gamma}$ with $\varphi\ge 0$. 

Moreover, we call $f\in {\mathcal H}(\gamma)$ an
$\form$-{\it harmonic} function 
with exponent $\gamma$ if it is both 
$\form$-subhamonic and $\form$-superharmonic with the same 
exponent $\gamma$. If we take $\gamma=1$, then we omit the phrase 
`with exponent $\gamma$'.
\end{definition}

\begin{theorem} \label{MT2} Let $1<p<2$ and $1/p \le \gamma \le 1$.
If $M_1<\infty$, then any non-negative $\form$-subharmonic 
function $f \in C^{\gamma} \cap L^p$ with exponent $\gamma$ is constant. 
\end{theorem} 
\begin{proof}
Let $f\in C^{\gamma} \cap L^p$ be a non-negative $\form$-subharmonic 
function 
with 
exponent $\gamma$. Take a sequence of functions 
$\{\chi_n\}\subset C_K^{\rm lip}$ as before:
\[
0 \le \chi_n \nearrow 1 \ {\rm as} \ n\to\infty \ 
{\rm and} \ |\chi_n(x)-\chi_n(y)|\le |x-y| \ \ {\rm for} \ \ x,y\in {\mathbf{R}}^d.
\]
Due to the subharmonicity of $f$, $\chi^2_n f^{p-1}\in C^{(p-1)\gamma}$, 
and the fact: $\gamma+(p-1)\gamma=\gamma p\ge 1$. we have that
\[
\form (f, \chi^2_n f^{p-1}) \le 0.
\]
Then by a similar argument in the proof of the previous theorem, using 
the derivation property (Corollary \ref{cor:deri}), it follows that
\[
0\le \int \chi^2_n(x) \Gamma(f,f^{p-1})(x)\, m(dx)
\le \Bigl| \int f^{p-1}(x) \Gamma(f,\chi_n^2)(x) \, m(dx) \Bigr|.
\]
The right-hand side of this inequality can be estimated as
\begin{align*}
& \dis{\Bigl| \int f^{p-1}(x) \Gamma(f,\chi_n^2)(x) \, m(dx) \Bigr|} \\  
\le & \dis{ \iint_{x\not=y} f^{p-1}(x) \bigl|f(x)-f(y)\bigr|
 \bigl|\chi^2_n(x)-\chi^2_n(f)\bigr| \mu(x,dy)m(dx) } \\ 
\le & \dis{ 2 \iint_{x\not=y} f^{p-1}(x) \bigl|f(x)-f(y)\bigr| 
\bigl|\chi_n(x)-\chi_n(f)\bigr| \mu(x,dy)m(dx) } \\ 
\le & \dis{ 2 \iint_{x\not=y} \left( f^{p-1}(x) 
\bigl|\chi_n(x)-\chi_n(f)\bigr|^{1/q} \right) } \\  
 & \dis{ \times \left(\bigl|f(x)-f(y)\bigr| 
\bigl|\chi_n(x)-\chi_n(f)\bigr|^{1/p} \right) \mu(x,dy)m(dx) } \\ 
\le & \dis{ 2 \left(\iint_{x\not=y} f^p(x) \bigl|\chi_n(x)-\chi_n(f)\bigr| 
\mu(x,dy)m(dx) \right)^{1/q} } \\  &  \dis{ \times 
\left(\iint_{x\not=y}  \bigl|f(x)-f(y)\bigr|^p  \bigl|\chi_n(x)-\chi_n(f) 
\bigr| \mu(x,dy)m(dx) \right)^{1/p}, } 
\end{align*}
where we used the H\"older inequality in the last inequality.  Using 
the symmetry of $ \mu(x,dy)m(dx)$ and the Fubini theorem, 
\begin{align*}
& \iint_{x\not=y} |f(x)-f(y)|^p {} 
\bigl|\chi_n(x)-\chi_n(y)\bigr| \mu(x,dy)m(dx) \\ 
& \le 2^p \iint_{x\not=y} |f(x)|^p |\chi_n(x)-\chi_n(y)\bigr| 
\mu(x,dy)m(dx) 
\end{align*}
Thus, we have:
\begin{align*}
0 & \le \int \chi_n(x)^2 \Gamma(f,f^{p-1})(x)\, m(dx) \\  
  & \le 4 \iint_{x\not=y} |f(x)|^p {} \bigl|\chi_n(x)-\chi_n(y)\bigr| 
\mu(x,dy)m(dx). 
\end{align*}
Applying the inequality:
\[
|f(x)|^p {} \bigl|\chi_n(x)-\chi_n(y)\bigr| 
 \le |f(x)|^p {} \bigl( 1\wedge |x-y| \bigr) 
\]
and taking into account that the right-hand side is integrable on 
${\mathbf{R}}^d\times{\mathbf{R}}^d \setminus D$ with respect to $\mu(x,dy)m(dx)$ 
because of the assumption $M_{1}<\infty$, the rest of the proof is 
similar to that of the preceding theorem. 
\end{proof}

\section{Examples} \label{E}
Let $k(x,y)=k(y,x)$ be a nonnegative measurable function defined on 
$x\not=y$. Set 
\[
\mu(x,dy)=k(x,y)m(dy)=k(x,y)dy,
\]
where $m(dx)=dx$ is the Lebesgue measure on ${\mathbf{R}}^d$. The examples 
for 
$\mu$ we present in the following have a density $k$, which we also call 
a kernel, and satisfy $M_2<\infty$ (see \cite{U04, U07}). 

\begin{example}[Symmetric stable-like process] 
Let $\alpha(x)$ be a measurable function defined on ${\mathbf{R}}^d$ satisfying
$\al \le \alpha (x) \le \beta$ with some constants $0<\al < \beta <2$.
The Dirichlet form $\form$ defined by the kernel $k$:
\[
\mu(x,dy)=k(x,y)dy=|x-y|^{-(d+\alpha(x))} dy
\]
corresponds to a symmetric stable-like process with variable exponent 
$\alpha(x)$. If $\al(x)=\alpha$ with $0<\alpha<2$, then the process is 
nothing but a {\em{(rotational invariant) symmetric $\alpha$-stable process}}. 
\end{example}

\begin{example}[Symmetric L\'evy processes]  
Let $\tilde{k}$ be an even positive measurable function on ${\mathbf{R}}^d$
so that 
\[
\int_{h \neq 0} \left( 1 \wedge |h|^2 \right) \tilde{k}(h)dh < \infty
\]
and set 
\[
k(x,y)=\tilde{k}(x-y),\quad x \neq y.
\]
Then $k$ satisfies $M_2<\infty$ and a symmetirc Hunt process associated with 
this kernel is a symmetric L\'evy process. 
\end{example}

\medskip
In the last example, we examine the $L^p$-Liouville property of a 
symmetric stable process:

\begin{example}[Symmetric stable process] 
The kernel $k$ 
\[
k(x,y)=|x-y|^{-(d+\alpha)}
\]
of a symmetric $\alpha$-stable process satisfies $M_{q}<\infty$ for 
$\alpha<q\le 2$.  Let us recall that the closure $\dom$ of 
$C_K^{\rm lip}$ with respect to the norm $\sqrt{\form(\cdot, \cdot)
+||\cdot||_{L^2}^2}$, where $\form$ is the Dirichlet form associated to $k$; 
that is,
\[
\form(f,g)=\iint_{x\not=y} \frac{(f(x)-f(y))(g(x)-g(y))}{|x-y|^{d+\alpha}} 
dxdy
\]
is nothing but the fractional Sobolev space $W^{s,2}({\mathbf{R}}^d)$ of order 
$s=\alpha/2$ $($and also coincides with the Besov space 
$B^{2,2}_{\alpha/2}({\mathbf{R}}^d))$. 

We state two cases $1\le \alpha<2$ and $0<\alpha<1$ 
separately:

\begin{itemize}
\item[{\rm (i)}] $1\le \alpha<2$. If $2\le p<\alpha/(\alpha-1)$, then any 
nonnegative $\form$-subharmonic function $f \in C \cap \dom_{\rm loc} 
\cap L^p$ must be constant.

\medskip

\item[{\rm (ii)}] $0<\alpha<1$. If $2\le p<\infty$, then
$M_q\le M_1<\infty$ with the conjugate number $q$ of $p$. 
Thus any nonnegative $\form$-subharmonic $f\in C \cap L^p$ is constant. 

\smallskip
Let $1<p<2$. Assume $1 / p \le \gamma \le 1$. Then a nonnegative H\"older 
continuous $\form$-subharmonic function $f$ with exponent $\gamma$ is 
constant if $f \in L^p$.
\end{itemize}
\end{example}



%
%

\end{document}